\documentclass{amsart}

\usepackage{ae}
\usepackage[all]{xy}
\usepackage{graphicx,amsfonts,amssymb,amsmath}
\usepackage{eurofont}
\usepackage{natbib}

\usepackage[latin1]{inputenc}
\usepackage[T1]{fontenc}
\usepackage[english]{babel}
\usepackage[cyr]{aeguill}

\newcommand{\chapeau}{{\rlap{\smash{\hbox{\lower4pt\hbox{\hskip1pt$\widehat{\phantom{u}}$}}}}}\mbox{ }}
\usepackage[OT2,T1]{fontenc}
\DeclareSymbolFont{cyrletters}{OT2}{wncyr}{m}{n}
\DeclareMathSymbol{\sha}{\mathalpha}{cyrletters}{"58}
\input cyracc.def

 \newtheorem{thm}{Theorem}[section]

 \newtheorem*{pb*}{\textit{Question}}
 \newtheorem*{cor*}{\textit{Corollary}}

 \newtheorem{lem}[thm]{Lemma}
 \newtheorem{prop}[thm]{Proposition}
 \theoremstyle{definition}
 
 \theoremstyle{remark}
 
 \theoremstyle{remark}

 \numberwithin{equation}{subsection}

 \newcommand{\To}{\longrightarrow}

 \newcommand{\coker}{\textup{Coker}}
 \renewcommand{\ker}{\textup{Ker}}
 \newcommand{\Br}{\textup{Br}}
 \newcommand{\CH}{\textup{CH}}

 \newcommand{\Q}{\mathbb{Q}}
 \newcommand{\Z}{\mathbb{Z}}

\usepackage[colorlinks]{hyperref}
\hypersetup{colorlinks,linkcolor=blue,citecolor=blue,}

\begin{document}

\title[]
{Local-global principle for certain biquadratic normic bundles}

\author{Yang CAO}
\author{Yongqi LIANG}

\address{Yang CAO \newline School of Mathematical Sciences, \newline Capital Normal University,
\newline 105 Xisanhuanbeilu, \newline 100048 Beijing, China}

\email{yangcao1988@gmail.com}

\address{Yongqi LIANG  \newline Institut de Mathématiques de Jussieu, \newline Université Paris Diderot - Paris VII, \newline 175 rue du Chevaleret,\newline  75013 Paris Cedex 13, France}

\email{liangy@math.jussieu.fr}



\date{\today.}



\maketitle

\begin{abstract}
Let $X$ be a proper smooth variety having an affine open subset defined by the normic equation
$N_{k(\sqrt{a},\sqrt{b})/k}({\textbf{x}})=Q(t_{1},\ldots,t_{m})^{2}$ over a number field $k$.
We prove that : (1) the failure of the local-global principle for zero-cycles is controlled by the Brauer group of $X;$
(2) the analogue for rational points is also valid assuming Schinzel's hypothesis.
\end{abstract}

\tableofcontents

\section{Introduction}

Let $k$ be a number field and $\Omega$ be the set of places of $k.$
We will discuss the local-global principle for rational points and for 0-cycles on algebraic varieties
$X$ supposed proper smooth and geometrically integral over $k.$ We write simply $X_{v}=X\times_{k}k_{v}$
for all $v\in\Omega$ and we denote by $\Br(X)$ the cohomological Brauer group of $X.$
For any abelian group $M,$ denote $\coker(M\buildrel{n}\over\to M)$ by $M/n.$

The so-called Brauer\textendash Manin obstruction to the local-global principle for rational points on $X$ is defined by Manin in 1970's
using the Brauer group $\Br(X),$ and
it is conjectured to be the only obstruction for geometrically rational varieties (or even a larger family of varieties) \cite{CTSansuc77-3}.
The local-global principle for 0-cycles is obstructed similarly.
After some reformulations, the following sequence $(E)$ is conjectured to be exact for all proper smooth
varieties
$$\varprojlim_n\CH_0(X)/n\to\prod_{v\in\Omega}\varprojlim_n\CH'_0(X_v)/n\to Hom(\Br(X),\mathbb{Q}/\mathbb{Z}),\leqno (E)$$
which means that the Brauer group gives the only obstruction to the local-global principle
for 0-cycles, \textit{cf.} \cite{CTSansuc81}, \cite{KatoSaito86}, \cite{CT95}
for more information about the conjecture, and \cite{Wittenberg} for more details on the sequence.
Instead of giving a very long list of references of contributions to this longtime focused question, we mention some recent papers
in which more historical details is presented :  \cite{B-HB} for rational points, \cite{Wittenberg} for 0-cycles, and \cite{Peyre} for older results.

In this paper, we restrict ourselves to a very concrete situation.
Let $K/k$ be a finite extension of degree $n$ and $P(t_{1},\ldots,t_{m})\in k[t_{1},\ldots,t_{m}]$ be a polynomial, the equation
$$N_{K/k}({\textbf{x}})=P(t_{1},\ldots,t_{m})$$
defines in $R_{K/k}\mathbb{A}^{1}\times\mathbb{A}^{m}=\mathbb{A}^{n+m}$ a closed subvariety fibered over $\mathbb{A}^{m}$ via
the parametric variables $t_{1},\ldots,t_{m}.$
We consider proper smooth models of such varieties.

\textit{\textbf{Question.}}
\textit{Is the Brauer\textendash Manin obstruction the only obstruction to the local-global principle for rational points and for 0-cycles
on this family of varieties?}

Very little is known when $K/k$
is a biquadratic Galois extension --- a ``simple'' case particularly mentioned in \cite[Rem. 1.5]{CT-Sk-SD}.
In such a case, theoretical difficulties come from not only the degeneracy of the fibers but also the non-triviality of Brauer groups of
the fibers.
We state the main result of the present paper as follows.

\begin{thm}\label{mainThm}
Let $k$ be a number field and $Q(t_{1},\ldots,t_{m})\in k(t_{1},\ldots,t_{m})$ be a non-zero rational function. Let $X$ be an arbitrary proper smooth model of the variety
defined by the equation
$$N_{K/k}({\textbf{x}})=Q(t_{1},\ldots,t_{m})^{2},$$
where $K/k$ is a biquardratic extension \textit{i.e.} a Galois extension of Galois group $\mathbb{Z}/2\mathbb{Z}\oplus\mathbb{Z}/2\mathbb{Z}.$
Then
\begin{itemize}
\item[(1)] the sequence $(E)$ is exact for $X;$
\item[(2)] assuming Schinzel's hypothesis, the Brauer\textendash Manin obstruction is the only obstruction to the Hasse principle
and to weak approximation for rational points on $X.$
\end{itemize}
\end{thm}
Our surprisingly  short proof (in \S \ref{proofsection}) bases on a geometric observation which permits us to reduce the question to a case that can be
deduced easily from existing results. In the last section \S \ref{generalizations}, we consider naturally expected generalizations
and explain why they can not be proved in the same manner.

\section{Proof of the theorem}\label{proofsection}
After several remarks, we will give a proof of the main theorem.
\subsection{Several preliminary remarks}

\subsubsection{Birational invariance of the question}
Let $X$ and $X'$ be proper smooth and geometrically integral $k$-varieties.
Suppose that they are birationally equivalent, \textit{i.e.} they have the same function field $k(X)=k(X').$ Then they have isomorphic
Brauer groups. It follows from Lang-Nishimura's theorem that the statement ``Brauer\textendash Manin obstruction is the only obstruction to the Hasse principle
and to weak approximation for rational points'' is valid for $X'$ as long as it is valid for $X.$ Moreover the Chow group $\CH_{0}(-)$ of 0-cycles
is also a birational invariant \cite[Prop. 6.3]{CT-Coray}, whence so is the exactness of the sequence $(E).$

\subsubsection{Algebraic tori associated to the equations}\label{tori}
Let $F$ be a field of characteristic $0.$ Any biquadratic extension $E$ of $F$ can be written as the form $E=F(\sqrt{a},\sqrt{b})$ with $a,b\in F^{*}\setminus F^{*2}.$ It has three different non-trivial subfields $F(\sqrt{a}),$ $F(\sqrt{b})$ and $F(\sqrt{ab})$ each of degree $2$ over $F.$

Let $T$ be the algebraic torus defined by the exact sequence of $F$-tori induced by the norm map
$$1\To T\To R_{E/F}\mathbb{G}_{m}\buildrel{N_{E/F}}\over\To\mathbb{G}_{m}\To1,$$
where $R_{E/F}$ is the Weil restriction of scalars.
Similarly, the multiplication of norm maps defines a torus $S$ fixed into the exact sequence
$$1\To S\To R_{F(\sqrt{a})/F}\mathbb{G}_{m}\times R_{F(\sqrt{b})/F}\mathbb{G}_{m}\times R_{F(\sqrt{ab})/F}\mathbb{G}_{m}\To\mathbb{G}_{m}\To1.$$
More explicitly, the tori $T$ and $S$ are defined respectively by
$$N_{E/F}({\textbf{x}})=1$$ and by
$$N_{F(\sqrt{a})/F}({\textbf{u}})\cdot N_{F(\sqrt{b})/F}({\textbf{v}})\cdot N_{F(\sqrt{ab})/F}({\textbf{w}})=1.$$
The multiplication $(\textbf{u},\textbf{v},\textbf{w})\mapsto \textbf{u}\cdot\textbf{v}\cdot\textbf{w}$ of the field $E$  induces
the middle vertical morphism of tori in the following diagram. Note that
$N_{E/F}(\textbf{u}\cdot\textbf{v}\cdot\textbf{w})=[N_{F_{1}/F}({\textbf{u}})\cdot N_{F_{2}/F}({\textbf{v}})\cdot N_{F_{3}/F}({\textbf{w}})]^{2},$
the square on the right is commutative, the left vertical morphism $\alpha$ is then induced.
\SelectTips{eu}{12}$$\xymatrix{
1\ar[r]&S\ar[r]\ar[d]^{\alpha}&R_{F(\sqrt{a})/F}\mathbb{G}_{m}\times R_{F(\sqrt{b})/F}\mathbb{G}_{m}\times R_{F(\sqrt{ab})/F}\mathbb{G}_{m}\ar[r]^-{\lambda}\ar[d]&\mathbb{G}_{m}\ar[r]\ar[d]^{2}&1
\\ 1\ar[r]&T\ar[r]&R_{E/F}\mathbb{G}_{m}\ar[r]^{\mu}&\mathbb{G}_{m}\ar[r]&1
}$$
\begin{lem}[{\cite[Prop. 3.1(b)]{CTnonpub}}]\label{keylemma}
The morphism $\alpha$ is an epimorphism whose kernel
$S_{0}$ is isomorphic to $\mathbb{G}_{m}^{2}.$
\end{lem}
\begin{proof}
We give a quick proof here and we will come back to the discussion of this key lemma in \S \ref{generalizations}, more proofs are
available.

We investigate the associated homomorphism between Galois modules of character groups.

Denote by $G\simeq\mathbb{Z}/2\mathbb{Z}\oplus\mathbb{Z}/2\mathbb{Z}$ the Galois group of $E/F.$
The subgroup $G_{a}=Gal(E/F(\sqrt{a}))\simeq\mathbb{Z}/2\mathbb{Z}$ (resp. $G_{b}=Gal(E/F(\sqrt{b})),$ $G_{ab}=Gal(E/F(\sqrt{ab}))$) is generated by an element $\tau$
(resp. $\sigma,$ $\tau\sigma$). The quotient $G/G_{a}\simeq\mathbb{Z}/2\mathbb{Z}$ (resp. $G/G_{b},$ $G/G_{ab}$) is generated by the image $\sigma'$ (resp. $\tau',$ $\gamma$)  of $\sigma$ (resp. $\tau,$ $\tau$).

With the notation above, we obtain the associated diagram of character groups
\SelectTips{eu}{12}$$\xymatrix{
0&\hat{S}\ar[l]&\mathbb{Z}\times\sigma'\mathbb{Z}\times\mathbb{Z}\times\tau'\mathbb{Z}\times\mathbb{Z}\times\gamma\mathbb{Z}\ar[l]&\mathbb{Z}\ar[l]_-{\hat{\lambda}}&0\ar[l]
\\ 0&\hat{T}\ar[l]\ar[u]^{\hat{\alpha}}&\mathbb{Z}\times\tau\mathbb{Z}\times\sigma\mathbb{Z}\times\tau\sigma\mathbb{Z}\ar[l]\ar[u]&\mathbb{Z}\ar[l]_-{\hat{\mu}}\ar[u]_{2}&0\ar[l]
}$$
where the homomorphism $\hat{\lambda}$ an $\hat{\mu}$ are diagonal embeddings and the middle vertical Galois equivariant  homomorphism is given explicitly by
$$(x_{1},x_{2},x_{3},x_{4})\mapsto(y_{1},y_{2},y_{3},y_{4},y_{5},y_{6})=(x_{1}+x_{2},x_{3}+x_{4},x_{1}+x_{3},x_{2}+x_{4},x_{1}+x_{4},x_{2}+x_{3}).$$
The Galois equivariant homomorphism
$$(y_{1},y_{2},y_{3},y_{4},y_{5},y_{6})\mapsto(y_{1}+y_{2}-y_{3}-y_{4},y_{1}+y_{2}-y_{5}-y_{6})$$ maps the diagonal image of $\mathbb{Z}$ to $0,$
hence it defines a homomorphism $\hat{\beta}:\hat{S}\to\mathbb{Z}\times\mathbb{Z}.$

One checks by explicit calculation that
$$0\To\hat{T}\buildrel^{\hat{\alpha}}\over\To\hat{S}\buildrel^{\hat{\beta}}\over\To\mathbb{Z}\times\mathbb{Z}\To0$$
is exact, which means that $\alpha$ is an epimorphism with kernel $S_{0}\simeq\mathbb{G}_{m}^{2}.$
\end{proof}
On the level of rational points, the torus $S_0$ is defined by the equations
$$\left\{
\begin{array}{l}
\textbf{u}\cdot\textbf{v}\cdot\textbf{w}=1\\
N_{F(\sqrt{a})/F}(\textbf{u})\cdot N_{F(\sqrt{b})/F}(\textbf{v})\cdot N_{F(\sqrt{ab})/F}(\textbf{w})=1
\end{array} \right..\leqno (\star)$$

In the proof of Theorem \ref{mainThm}, the field $F$ will be the function field $k(t_{1},\ldots,t_{m}),$ the tori $T$ and $S$ will be isotrivial \textit{i.e.}
$a,b\in k^{*}$ and $E=F(\sqrt{a},\sqrt{b})=k(\sqrt{a},\sqrt{b})(t_{1},\ldots,t_{m}).$
By abuse of notation, we denote the $k$-torus
$N_{k(\sqrt{a},\sqrt{b})/k}({\textbf{x}})=1$ also
by $T,$ and similarly for $S.$

\subsection{Proof of Theorem \ref{mainThm}}\label{proof}

\begin{proof}
We may write $K=k(\sqrt{a},\sqrt{b})$ with $a,b\in k^{*}.$
The variety $X$ that we consider is a proper smooth model of the equation $N_{K/k}({\textbf{x}})=Q(t_{1},\ldots,t_{m})^{2}.$
With the notation in \S\ref{tori}, this equation defines the fiber of $\mu$ over the point $Q(t_{1},\ldots,t_{m})^{2}$ of $\mathbb{G}_{m,k(t_{1},\ldots,t_{m})},$
we denote it by $W.$ It is a principal homogeneous space under the torus $T.$

Consider the fiber, denoted by $V,$ of $\lambda$ over the point $Q(t_{1},\ldots,t_{m})$  of $\mathbb{G}_{m,k(t_{1},\ldots,t_{m})}.$ It is defined by the equation
$$N_{k(\sqrt{a})/k}({\textbf{u}})\cdot N_{k(\sqrt{b})/k}({\textbf{v}})\cdot N_{k(\sqrt{ab})/k}({\textbf{w}})=Q(t_{1},\ldots,t_{m}),$$ it is
a principal homogeneous space under the torus $S.$

Associating $(\textbf{u},\textbf{v},\textbf{w})$ to their product $\textbf{x}=\textbf{u}\cdot\textbf{v}\cdot\textbf{w},$
we define a $k(t_{1},\ldots,t_{m})$-morphism $\phi:V\to W.$ The last morphism extends to a certain $k$-morphism
$\Phi$ between $k$-varieties
\SelectTips{eu}{12}$$\xymatrix{
V_{0}\ar[rr]^{\Phi}\ar[dr]&&W_{0}\ar[dl]
\\ &U&
}$$
fibered over a certain open subset $U$ of $\mathbb{P}^{m},$ whose generic fiber is the morphism $\phi.$
As $S\to T$ is an epimorphism of tori, torsors $V$ and $W$ become trivial over the algebraic closure, the morphism
$\phi$ is then geometrically surjective. We may assume moreover that $\Phi$ is geometrically surjective.
The morphism $\Phi:V_{0}\to W_{0}$ is a torsor under the torus $S_{0}.$
In fact, the fiber of $\Phi$ over each point $(\textbf{x},t_{1},\ldots,t_{m})$ of $W_{0}$ is defined by the equations
$$\left\{
\begin{array}{l}
\textbf{u}\cdot\textbf{v}\cdot\textbf{w}=\textbf{x}\\
N_{F(\sqrt{a})/F}(\textbf{u})\cdot N_{F(\sqrt{b})/F}(\textbf{v})\cdot N_{F(\sqrt{ab})/F}(\textbf{w})=Q(t_1,\ldots,t_m)
\end{array} \right.,$$
on which the torus $S_0$ (defined by $(\star)$ in \S\ref{tori})
acts freely transitively by multiplication at each coordinate component.
In other words, the variety $V_{0}$ defines a class in the cohomology $H^{1}(W_{0},S_0).$
Note that $S_0\simeq\mathbb{G}_{m}^{2}$ by Lemma \ref{keylemma},
thanks to Hilbert's 90 we obtain $H^{1}(k(W_{0}),S_0)=0.$ Restricted to the generic point $Spec(k(W_{0}))$ of $W_{0},$ the class $[V_{0}]$ becomes $0.$
Therefore the function field $k(V_{0})$ is a purely transcendental extension of $k(W_{0})$ of transcendental degree
$2.$ We deduce that $V_{0}$ is birationally equivalent to $W_{0}\times\mathbb{P}^{2},$ and the latter is birationally equivalent
to $X\times\mathbb{P}^{2}.$

Recall that $V_{0}\to U\subset\mathbb{P}^{m}$ has generic fiber defined by
$$N_{k(\sqrt{a})/k}({\textbf{u}})\cdot N_{k(\sqrt{b})/k}({\textbf{v}})\cdot N_{k(\sqrt{ab})/k}({\textbf{w}})=Q(t_{1},\ldots,t_{m}).$$
Birationally, the equation is the same as
$$N_{k(\sqrt{a})/k}({\textbf{u}})=\frac{Q(t_{1},\ldots,t_{m})}{N_{k(\sqrt{b})/k}({\textbf{v}})\cdot N_{k(\sqrt{ab})/k}({\textbf{w}})},$$
which can be viewed as a fibration in conics over $\mathbb{P}^{m+4}$ via the parametric variables $(t_{1},\ldots,t_{m},\textbf{v},\textbf{w}).$
For proper smooth models of $V_{0},$
the statement (1) of the theorem has been proved  in \cite[\S6]{Liang3} (with a Corrigendum); and
the statement (2) has been proved by Wittenberg in \cite[Cor. 3.5]{WittenbergLNM}; an alternative proof for the case $m=1$
is also available in a recent preprint of Wei \cite[Thm. 3.5]{Wei}.

By the birational invariance of the statements (1) and (2), the following lemma will complete the proof.
\end{proof}

\begin{lem}
Let $k$ be a number field and $X$ be a proper smooth and geometrically integral $k$-variety.
\begin{itemize}
\item[(1)] If the sequence $(E)$ is exact for $X\times\mathbb{P}^{n},$ then it is also exact for $X.$
\item[(2)] If the Brauer\textendash Manin obstruction to the Hasse principle and to weak approximation for rational points on $X\times\mathbb{P}^{n},$ then it is also the case for $X.$
\end{itemize}
\end{lem}

\begin{proof}
Note that the projection $\pi:X\times\mathbb{P}^{n}\to X$ induces an isomorphism
$\Br(X)\buildrel{\simeq}\over\to \Br(X\times\mathbb{P}^{n}).$ Fix a $k$-rational point $p$  of $\mathbb{P}^{n},$ the map
$x\mapsto (x,p)$ defines a section $\sigma:X\to X\times\mathbb{P}^{n}$ of $\pi.$ Diagram chasing proves the first statement.
Functoriality of the Brauer\textendash Manin pairing and obvious fibration argument prove the second statement.
\end{proof}

\section{A remark on ``generalizations''}\label{generalizations}

In this section, we consider naturally expected generalizations
and explain why they can not be proved in the same manner.

Let $p$ be a prime number and $K/k$ be a Galois extension of Galois group $\mathbb{Z}/p\mathbb{Z}\oplus\mathbb{Z}/p\mathbb{Z}.$
One may expect that the analogue of Theorem \ref{mainThm} holds for proper smooth models of the equation
$$N_{K/k}({\textbf{x}})=Q(t_{1},\ldots,t_{m})^{p}.$$
Once an analogue of Lemma \ref{keylemma} for more general $p$ is established,
\textit{i.e.} the homomorphism $\alpha$ is an epimorphism with kernel $S_0\simeq\mathbb{G}_m^p,$
all the remaining arguments still work well.
In this section we will prove Proposition \ref{keyprop} below.
Unfortunately, for $p>2$ the kernel $S_0$ is not connected anymore.

Consider a Galois extension $E/F$ of Galois group $G=\mathbb{Z}/p\mathbb{Z}\oplus\mathbb{Z}/p\mathbb{Z},$
it has $p+1$ nontrivial sub-extensions $F_{i}(i=0,\ldots,p).$ The subgroup $G_{i}=Gal(E/F_{i})\simeq\mathbb{Z}/p\mathbb{Z}$ and
$H_{i}=G/G_{i}=Gal(F_{i}/F)\simeq\mathbb{Z}/p\mathbb{Z}.$
We can write down explicitly the Galois equivariant homomorphism induced by quotients $G\to H_i$
$$\hat{\rho}:\mathbb{Z}[G]\to\prod_{i=0}^{p}\mathbb{Z}[H_{i}].$$
It factorises through quotients by the diagonally embedded $\mathbb{Z}$ and gives the homomorphism
$$\hat{\alpha}:\hat{T}=\frac{\mathbb{Z}[G]}{\mathbb{\mathbb{Z}}}\to\hat{S}=\frac{\prod_{i=0}^{p}\mathbb{Z}[H_{i}]}{\mathbb{Z}}.$$
This last homomorphism between Galois modules is associated to a morphism $\alpha:S\to T$ between algebraic tori, where $S$ is the torus defined by
$\prod_{i=0}^pN_{F_i/F}(\textbf{u}_\textbf{i})=1$
and $T$ is the torus defined by $N_{E/F}(\textbf{x})=1.$

\begin{prop}\label{keyprop}
The morphism $\alpha$ is an epimorphism. The kernel $S_0=\ker(\alpha)$ is a group of multiplicative type, its
identity component $S_0^\circ$ is isomorphic to $\mathbb{G}_m^p$ and its group of connected component $\pi_0(S_0)$
is isomorphic to $(\Z/p\Z)^{p-2}$ as an abelian group. In particular $S_0\simeq\mathbb{G}_m^p$ for $p=2.$

\end{prop}

The proof by investigating exact sequences of Galois modules
is not difficult but rather long, we outline the main steps and leave some detailed verification to the readers.
We will apply repeatedly the following obvious lemma.

\begin{lem}\label{obviouslemma}
Let $R$ be a commutative ring and let $I$ and $J$ be its ideals. Then $r\mapsto (r,r)$ and $(r_1,r_2)\mapsto r_1-r_2$ give
rise to an exact sequence of $R$-modules
$$R/IJ\To R/I\times R/J\To R/(I+J)\To0.$$
\end{lem}

\begin{proof}[Proof of Proposition \ref{keyprop}]
We begin with a commutative diagram with exact rows
$$\xymatrix{
0\ar[r]&\mathbb{Z}\ar[r]\ar[d]^{p}&\mathbb{Z}[G]\ar[r]\ar[d]^{\hat{\rho}}&
\hat{T}\ar[r]\ar[d]^{\hat{\alpha}}&0
\\ 0\ar[r]&\mathbb{Z}\ar[r]&\prod_{i=0}^p\mathbb{Z}[H_i]\ar[r]&\hat{S}\ar[r]&0
}.\leqno{(\sharp 1)}$$
Note that $\hat{T}$ is a free abelian group, the snake lemma implies that the injectivity of $\hat{\alpha}$
is equivalent to that of $\hat{\rho}.$ It suffices to prove the injectivity and study the cokernel of $\hat{\alpha}$
via the induced exact sequence
$$0\To\Z/p\Z\To \coker(\hat{\rho})\To \coker(\hat{\alpha})\To0.\leqno{(\sharp 2)}$$

To fix notation, we choose generators $x$ and $y$ of $G=\Z/p\Z\oplus\Z/p\Z.$
With multiplicative notation, the group ring $\Z[G]$ is isomorphic to the quotient of the polynomial ring $\Z[x,y]/(x^p-1,y^p-1)$
equipped with obvious Galois action via multiplication.
Similarly, we have $\Z[H_i]\simeq\Z[z_i]/(z_i^p-1)$ and $\Z[G]\to\Z[H_i]$ is given by
$x\mapsto z_i,y\mapsto z_i^i$ (resp. $x\mapsto 1,y\mapsto z_0$) if  $i\neq0$ (resp. $i=0$).
We denote $1+t+\cdots+t^{p-1}$ simply by $v(t)$ for $t=x,y,$ or $z_i.$ Then we have the following
natural homomorphisms\\
\noindent$h_i:\Z[H_i]\simeq\Z[z_i]/(z_i^p-1)\to\Z[z_i]/(v(z_i))\times\Z[z_i]/(z_i-1)=\Z[z_i]/(v(z_i))\times\Z,$\\
\noindent$f_i:\Z[G]\to\Z[H_i]\buildrel{h_i}\over\to\Z[z_i]/(v(z_i))\times\Z,$\\
\noindent$g_i=pr_i^1\circ f_i:\Z[G]\to\Z[z_i]/(v(z_i)),$\\
\noindent$j=pr_i^2\circ f_i:\Z[G]\to \Z.$\\
We remark that this last homomorphism $j$ does not depend on $i.$

\medskip
\noindent\textit{Injectivity of $\hat{\rho}.$}

As $\Z[G]$ is torsion free, it suffices to show that $\hat{\rho}\otimes\Q$ is injective. It is clear that
$h_{i\Q}$ maps $\Q[H_i]$ isomorphically to a product of two fields $\Q[z_i]/(v(z_i))\times\Q.$
By Maschke's theorem and Artin\textendash Wedderburn theorem, the group algebra $\Q[G]$  of an abelian group
is also isomorphic to a product of finite extensions of $\Q,$ hence $Spec(\Q[G])$ is a finite disjoint union of its reduced closed points.

Consider reduced closed subschemes of $Spec(\Q[G])$ given by surjective projections $f_{i\Q}:\Q[G]\twoheadrightarrow\Q[z_i]/(v(z_i))\times\Q,$
all of them contain a common reduced closed point given by $j_{\Q}:\Q[G]\twoheadrightarrow\Q.$
Once we can show that $g_{i\Q}:\Q[G]\twoheadrightarrow\Q[z_i]/(v(z_i))$ have distinct kernels for different $i,$
then $Spec(\Q[z_i]/(v(z_i)))$ are distinct closed points of $Spec(\Q[G]).$
By calculation of dimension of $\Q$-vector spaces, we can conclure that $\Q[G]$ is isomorphic to $\prod_{i=0}^p\Q[z_i]/(v(z_i))\times\Q.$

Note that $\ker(f_{i\Q})\subset \ker(g_{i\Q})$ and $dim_{\Q}\ker(g_{i\Q})=dim_{\Q}\ker(f_{i\Q})+1=p(p-1)+1,$
in order to prove that $\ker(g_{i\Q})\neq \ker(g_{i'\Q})$ for $i\neq i'$ it suffices to show that
$dim_{\Q}(\ker(f_{i\Q})\cap \ker(f_{i'\Q}))\leqslant p(p-1)-2.$ By the symmetricity of subgroups of $G$ we may assume that $i=0$ and
$i'=p,$ simple calculation shows that the inequality is valid if $p>2.$
For the case $p=2,$ we can check directly without dimensional argument that $\ker(g_{i\Q})\neq \ker(g_{i'\Q})$ for $i\neq i'.$

We have proved the injectivity of $\hat{\rho}$ and $\hat{\alpha}.$

\smallskip
\noindent\textit{Calculation of $\coker(\hat{\alpha}).$}

Consider the following commutative diagram with exact rows
$$\xymatrix@C=14pt{
0\ar[r]&\mathbb{Z}[G]\ar[r]^-{\hat{\rho}}\ar[d]^{\prod g_i\times j}&\prod_{i=0}^p\mathbb{Z}[H_i]\ar[r]\ar[d]^{\prod h_i} &\coker(\hat{\rho})\ar[r]\ar[d]^{\varphi}&0\\
0\ar[r]&\prod_{i=0}^p\mathbb{Z}[z_i]/(v(z_i))\times \mathbb{Z}\ar[r]&\prod_{i=0}^p\mathbb{Z}[z_i]/(v(z_i))\times \prod_{i=0}^p\mathbb{Z}\ar[r]& \mathbb{Z}^{p+1}/\Z\ar[r]&0
},\leqno{(\sharp 3)}$$
where $\Z^{p+1}/\Z$ is the cokernel of the diagonal embedding.
The induced homomorphism $\varphi$ is surjective since $\Z[H_i]\to\Z$ is surjective for every $i.$ We have proved that
$(\prod g_i\times j)_\Q$ and $(\prod h_i)_\Q$ are isomorphisms, then $\prod g_i\times j$ and $\prod h_i$
are both injective and $\ker(\varphi)$ is a torsion group.
Lemma \ref{obviouslemma} shows that $\coker(\prod h_i)=(\Z/p\Z)^{p+1}.$
We obtain an exact sequence
$$0\To \ker(\varphi)\To \coker(\prod g_i\times j)\To (\Z/p\Z)^{p+1}\To0$$
On the other hand, in the sequence $(\sharp 2)$ the image of $\Z/p\Z$ must be contained in the torsion part $\coker(\hat{\rho})_{tors}=\ker(\varphi).$
We denote by $\hat{\pi}_0=\ker(\varphi)/(\Z/p\Z),$ then we obtain exact sequences
$$0\To \hat{\pi}_0\To \frac{\coker(\prod g_i\times j)}{\Z/p\Z}\To (\Z/p\Z)^{p+1}\To0\leqno{(\sharp 4)}$$ and
$$0\To \hat{\pi}_0\To \coker(\hat{\alpha})\To\Z^{p+1}/\Z\To0.$$
Since $\Z^{p+1}/\Z$ is isomorphic to $\Z^p$ as a Galois module, in order to conclure it suffices to show that
$\hat{\pi}_0$ is isomorphic to $(\Z/p\Z)^{p-2}$ as an abelian group.

The rest of the proof is devoted to the calculation of $\hat{\pi}_0.$ Lemma \ref{obviouslemma} gives the following two exact
sequences (who become the middle column of the next diagram)
$$0\to \mathbb{Z}[x,y]/(x^p-1,y-1)\to \mathbb{Z}[x,y]/(v(x),y-1)\times \mathbb{Z}[x,y]/(x-1,y-1)\to \mathbb{Z}/p\Z\to 0,$$
$$0\to \mathbb{Z}[x,y]/(x^p-1,v(y))\to \mathbb{Z}[x,y]/(v(x),v(y))\times \mathbb{Z}[x,y]/(x-1,v(y))\to \mathbb{F}_p[y]/(v(y))\to 0,$$
where the injectivity on the left is verified by tensoring with $\Q$.
Then we obtain the following commutative diagram by applying the snake lemma to the first two columns.
The exactness of the upper row is deduced again by Lemma \ref{obviouslemma}. The first homomorphism in the middle row is nothing but
$\prod g_i\times j$ seen by identifying $x,y$ with their images in $\Z[z_i]/(v(z_i))$ for each $i.$
$$\xymatrix@C=10pt{
&0\ar[d]&0\ar[d]&\\
\mathbb{Z}[G]\ar[r]\ar[d]^=&\mathbb{Z}[x,y]/(x^p-1,v(y))\times\mathbb{Z}[x,y]/(x^p-1,y-1)\ar[r]\ar[d]^\psi&\mathbb{F}_p[x]/(x^p-1)\ar[r]\ar[d]^{\bar{\psi}}&0\\
\mathbb{Z}[G]\ar[r]&\mathbb{Z}[x,y]/(v(x),v(y))\times \mathbb{Z}[y]/(v(y))\times \mathbb{Z}[x]/(v(x))\times \mathbb{Z}\ar[r]\ar[d]^\eta&\coker(\prod g_i\times j)\ar[r]\ar[d]^{\bar{\eta}}&0\\
&\mathbb{F}_p[y]/(v(y))\times \mathbb{Z}/p\Z\ar[d]\ar[r]^=&\mathbb{F}_p[y]/(v(y))\times \mathbb{Z}/p\Z\ar[d]&\\ &0&0&
}.\leqno{(\sharp 5)}$$

We are going to find out the image of $1\in \Z/p\Z$ in $\ker(\varphi)\subset \coker(\prod g_i\times j).$
According to the diagram $(\sharp 1),$ its image in $\coker(\hat{\rho})$ lifts to the element of $\prod_{i=0}^p\Z[H_i]$
represented by $(v(z_i))_i\in\prod_{i=0}^p\Z[z_i]/(z_i^p-1)=\prod_{i=0}^p\Z[H_i].$
Applying $\prod h_i$ we get $(0,\ldots,0,p,\ldots,p)\in\prod_{i=0}^p\mathbb{Z}[z_i]/(v(z_i))\times \prod_{i=0}^p\mathbb{Z},$
which comes from $(0,\ldots,0,p)\in\prod_{i=0}^p\mathbb{Z}[z_i]/(v(z_i))\times\mathbb{Z}.$
This last element is identified (in $\sharp 5$) with $(0,0,0,p)\in\mathbb{Z}[x,y]/(v(x),v(y))\times \mathbb{Z}[y]/(v(y))\times \mathbb{Z}[x]/(v(x))\times \mathbb{Z},$
which is exactly the image of $(0,v(x))$ under $\psi.$ Therefore the image of $1\in\Z/p\Z$ in $\coker(\prod g_i\times j)$
equals $-v(x)\in\mathbb{F}_p[x]/(x^p-1)\subset \coker(\prod g_i\times j)$ and hence
$$0\To\mathbb{F}_p[x]/(v(x))\buildrel{\bar{\bar{\psi}}}\over\To\frac{\coker(\prod g_i\times j)}{\Z/p\Z}\buildrel{\bar{\eta}}\over\To\mathbb{F}_p[y]/(v(y))\times \mathbb{Z}/p\Z\To0$$
is exact.

We are going to show that this exact sequence is split.
Consider the element $x\in\mathbb{F}_p[x]/(v(x))$ and its image $\bar{\bar{\psi}}(x).$ The element $x$ lifts to $(0,-x)$
in the middle column of the diagram $(\sharp 5),$ then $\bar{\bar{\psi}}(x)=\bar{\psi}(x)$ lifts to $\psi(0,-x)=(0,0,-x,-x)=(0,0,-x,-1).$
On the other hand, we have another diagram $(\sharp 5')$ switching $x$ and $y$ with mappings denoted by $\psi',$ $\eta',$ $\bar{\psi}',$ and $\bar{\eta}'.$
In the middle row of $(\sharp 5'),$ we have $(0,-x,0,-1)\mapsto\bar{\psi}(x)\in \coker(\prod g_i\times j),$
hence $\bar{\eta}'(\bar{\psi}(x))=\eta'(0,-x,0,-1)=(x,1)\in\mathbb{F}_p[x]/(v(x))\times \mathbb{Z}/p\Z.$
In other words $pr^1(\bar{\eta}'(\bar{\bar{\psi}}(x)))=x\in\mathbb{F}_p[x]/(v(x))$ where $pr^1$ is the projection to the first component.
Notice that all the homomorphisms are Galois equivariant and the Galois action on $\mathbb{F}_p[x]/(v(x))$ is given by multiplication,
therefore $pr^1\circ\bar{\eta}'$ is a splitting of $\bar{\bar{\psi}}$ and

$$\frac{\coker(\prod g_i\times j)}{\Z/p\Z}\simeq \mathbb{F}_p[x]/(v(x))\times\mathbb{F}_p[y]/(v(y))\times \mathbb{Z}/p\Z.$$
Then the sequence $(\sharp 4)$ implies that $\hat{\pi}_0$ is isomorphic to $(\Z/p\Z)^{p-2}$ as an abelian group. This completes the proof.
\end{proof}

In the end, we would like to give an alternative proof of Lemma \ref{keylemma}.
Let us follow the tracks of the proof which may explain to some extent why we have the difference on connectedness
between the cases $p=2$ and $p>2:$
the equivalences $(\star)\Leftrightarrow(\star\star)\Leftrightarrow(\star\star\star)$ can be easily established in the case $p=2,$ but
difficulties appear for $p>2.$

\begin{proof}[An alternative proof of Lemma \ref{keylemma}]
On the level of rational points, for any extension $L$ of $F$ we need to show that $S_{0}(L)$ is (functorially) isomorphic to $\mathbb{G}_{m}^{2}(L)$ as abelian groups.

By definition $S_{0}(L)$ is given by the triples
$$(\textbf{u},\textbf{v},\textbf{w})\in F(\sqrt{a})\otimes_{F}L\times F(\sqrt{b})\otimes_{F}L\times F(\sqrt{ab})\otimes_{F}L$$
satisfying
$$\left\{
\begin{array}{l}
\textbf{u}\cdot\textbf{v}\cdot\textbf{w}=1\\
N_{F(\sqrt{a})/F}(\textbf{u})\cdot N_{F(\sqrt{b})/F}(\textbf{v})\cdot N_{F(\sqrt{ab})/F}(\textbf{w})=1
\end{array} \right..\leqno (\star)$$
Fix a $L$-linear base of $F(\sqrt{a})\otimes_{F}L$ (resp. $F(\sqrt{b})\otimes_{F}L,$ $F(\sqrt{ab})\otimes_{F}L$), we write
$\textbf{u}=u_{1}+u_{2}\sqrt{a}$ (resp. $\textbf{v}=v_{1}+v_{2}\sqrt{b},$ $\textbf{w}=w_{1}+w_{2}\sqrt{ab}$) with
$u_{1}, u_{2},v_{1},v_{2},w_{1},w_{2}\in L.$
Easy calculation shows that $(\star)$ is equivalent to
$$\left\{
\begin{array}{l}
1=u_{1}v_{1}w_{1}\\
0=u_{2}v_{2}w_{2}\\
0=u_{1}v_{2}w_{2}b+u_{2}v_{1}w_{1}\\
0=u_{1}v_{2}w_{1}+u_{2}v_{1}w_{2}a\\
0=u_{1}v_{1}w_{2}+u_{2}v_{2}w_{1}
\end{array} \right..\leqno(\star\star)$$
Whence one of $u_{2},v_{2},w_{2}$ must be $0,$ and no matter which one equals $0$ the last three equalities imply that the other two are also $0.$
Then $(\star\star)$ is equivalent to
$$\left\{
\begin{array}{l}
1=u_{1}v_{1}w_{1}\\
0=u_{2}=v_{2}=w_{2}\\
\end{array} \right.,\leqno(\star\star\star)$$
which by definition is exactly $\mathbb{G}_{m}^{2}(L).$
\end{proof}


\bibliographystyle{alpha}
\bibliography{mybib1}

\begin{thebibliography}{CTSSD98}

\bibitem[BHB12]{B-HB}
T.~D. Browning and D.~R. Heath-Brown.
\newblock Quadratic polynomials represented by norm forms.
\newblock {\em Geometric and Functional Analysis}, 22:1124--1190, 2012.

\bibitem[CT]{CTnonpub}
J.-L. Colliot-Th\'el\`ene.
\newblock Groupe de {B}rauer non ramifi\'e d'espaces homog\`enes de tores.
\newblock To appear in Proc. Amer. Math. Soc., available at arXiv:1201.1815.

\bibitem[CT95]{CT95}
J.-L. Colliot-Th\'el\`ene.
\newblock L'arithm\'etique du groupe de {C}how des z\'ero-cycles.
\newblock {\em J. Th\'eorie de nombres de Bordeaux}, 7:51--73, 1995.

\bibitem[CTC79]{CT-Coray}
J.-L. Colliot-Th\'el\`ene and D.~Coray.
\newblock L'\'equivalence rationnelle sur les points ferm\'es des surfaces
  rationnelles fibr\'ees en coniques.
\newblock {\em Compositio Math.}, 39:301--332, 1979.

\bibitem[CTS77]{CTSansuc77-3}
J.-L. Colliot-Th\'el\`ene and J.-J. Sansuc.
\newblock La descente sur une vari\'et\'e rationnelle d\'efinie sur un corps de
  nombres.
\newblock {\em C.R.A.S. Paris}, 284:1215--1218, 1977.

\bibitem[CTS81]{CTSansuc81}
J.-L. Colliot-Th\'el\`ene and J.-J. Sansuc.
\newblock On the {C}how groups of certain rational surfaces : a sequel to a
  paper of {S}.{B}loch.
\newblock {\em Duke Math. J.}, 48:421--447, 1981.

\bibitem[CTSSD98]{CT-Sk-SD}
J.-L. Colliot-Th\'el\`ene, A.N. Skorobogatov, and Sir~Peter Swinnerton-Dyer.
\newblock Rational points and zero-cycles on fibred varieties : {S}chinzel's
  hypothesis and {S}alberger's device.
\newblock {\em J. reine angew. Math.}, 495:1--28, 1998.

\bibitem[KS86]{KatoSaito86}
K.~Kato and S.~Saito.
\newblock Global class field theory of arithmetic schemes.
\newblock {\em Contemporary Math.}, 55:255--331, 1986.

\bibitem[Lia]{Liang3}
Y.~Liang.
\newblock Astuce de {S}alberger et z\'ero-cycles sur certaines fibrations.
\newblock To appear in International Mathematics Research Notices, available
  online:\\ http://imrn.oxfordjournals.org/content/early/2012/02/01/imrn.rns003
  Corrigendum available at \\ http://www.math.jussieu.fr/\textasciitilde
  liangy/files/recherche.htm.

\bibitem[Pey]{Peyre}
E.~Peyre.
\newblock Obstructions au principe de {H}asse et \`a l'approximation faible.
\newblock S\'eminaire Bourbaki Vol. 2003/2004.

\bibitem[Wei]{Wei}
D.~Wei.
\newblock On the equation ${N}_{K/k}({\Xi})={P}(t)$.
\newblock Preprint, available at arXiv:1202.4115.

\bibitem[Wit07]{WittenbergLNM}
O.~Wittenberg.
\newblock {\em Intersections de deux quadriques et pinceaux de courbes de genre
  1}, volume 1901 of {\em Lecture Notes in Mathematics}.
\newblock Springer, 2007.

\bibitem[Wit12]{Wittenberg}
O.~Wittenberg.
\newblock Z\'ero-cycles sur les fibrations au-dessus d'une courbe de genre
  quelconque.
\newblock {\em J. Duke Math.}, 161:2113--2166, 2012.

\end{thebibliography}
\end{document}